\documentclass[a4paper, 11pt]{amsart}

\usepackage[T1]{fontenc}
\usepackage[utf8]{inputenc}

\usepackage{amssymb}

\makeatletter
\@namedef{subjclassname@2020}{%
  \textup{2020} Mathematics Subject Classification}
\makeatother

\allowdisplaybreaks[2]

\newtheorem{theorem}{Theorem}[section]
\newtheorem{proposition}[theorem]{Proposition}
\newtheorem{lemma}[theorem]{Lemma}
\newtheorem{corollary}[theorem]{Corollary}
\theoremstyle{definition}
\newtheorem{remark}[theorem]{Remark}

\newtheorem{example}[theorem]{Example}

\numberwithin{equation}{section}

\numberwithin{equation}{section}

\DeclareMathOperator{\tri}{Tri}

\begin{document}

\title[Weighted Jordan homomorphisms]{Weighted Jordan homomorphisms}
\author{M. Bre\v sar}

\address{M. Bre\v sar, Faculty of Mathematics and Physics,  University of Ljubljana,
 and Faculty of Natural Sciences and Mathematics, University 
of Maribor, Slovenia}
 \email{matej.bresar@fmf.uni-lj.si}

\author{M. L. C. Godoy}
\address{M. L. C. Godoy, Departamento de An\' alisis
Matem\' atico, Fa\-cul\-tad de Ciencias, Universidad de Granada,
 Granada, Spain} 
 \email{mgodoy@ugr.es}

\thanks{The first author was supported by the Slovenian Research Agency (ARRS) Grant P1-0288. The second author was suported by MCIU/AEI/FEDER Grant PGC2018-093794-B-I00, Junta de Andaluc\' ia Grant FQM-185, MIU Grant FPU18/00419 and MIU Grant EST19/00466.}



\begin{abstract}
Let $A$ and $B$  be   unital rings. An additive map $T:A\to B$ is called a  weighted Jordan homomorphism if $c=T(1)$ is an invertible central  element  and $cT(x^2) = T(x)^2$ for all $x\in A$.
 We provide assumptions, which are in particular fulfilled when $A=B=M_n(R)$ with $n\ge 2$ and  $R$  any unital ring with $\frac{1}{2}$, under which every surjective additive map $T:A\to B$ with the property that $T(x)T(y)+T(y)T(x)=0$ whenever $xy=yx=0$ is a  weighted Jordan homomorphism. Further, we show that if $A$ is a prime ring with {\rm char}$(A)\ne 2,3,5$, then a bijective additive map $T:A\to A$ is a weighted Jordan homomorphism provided that there exists an additive map $S:A\to A$ such that
  $S(x^2)=T(x)^2$ for all $x\in A$.
\end{abstract}
\maketitle

\section{Introduction}\label{s1}

Let $A$ and $B$ be  unital rings. 
 Recall that an additive map $\Phi:A\to B$ is called a {\em Jordan homomorphism} if $\Phi(x\circ y)=\Phi(x)\circ \Phi(y)$ for all $x,y\in A$, where $x\circ y$ stands for the Jordan product $xy+yx$ of  $x$ and $y$.
We  say that an additive map $T:A\to B$ is a {\em weighted Jordan homomorphism} if $c=T(1)$ is an invertible element lying in the center  of $B$ and $x\mapsto c^{-1}T(x)$ is a Jordan homomorphism, that is,
$$cT(x\circ y) = T(x)\circ T(y)\quad(x,y\in A).$$
 Weighted Jordan homomorphisms can be also defined for rings without unity, see \cite[p.~121]{zpdbook}. However, we will work only with unital rings in this paper.
 
  Weighted Jordan homomorphisms naturally appear in some preserver problems. In \cite{C}, Chebotar, Ke, Lee, and Zhang used functional identities to prove that if
  $R$ is a  unital ring with $\frac{1}{2}$ and  $A=M_n(R)$ with $n\ge 4$ (i.e., $A$ is the ring of $n\times n$ matrices over $R$), then a surjective additive map $T:A\to A$
  which preserves zero Jordan products (i.e.,  $T(x)\circ T(y)=0$ whenever $x\circ y=0$)  is 
 a weighted Jordan homomorphism. We also mention more recent papers \cite{Cat, CHK} which are close to \cite{C} and  also involve weighted Jordan homomorphisms.
  Further,  
 Alaminos,  Bre\v sar, Extremera, and Villena   \cite{ABEV1}  proved that if $A$ and $B$ are $C^*$-algebras and $T:A\to B$ is a continuous linear map with the property that  for all $x,y\in A$, 
 \begin{equation}\label{condf} xy=yx=0\implies T(x)\circ T(y)=0,\end{equation} 
 then $T$ is a weighted Jordan homomorphism.
The proof was based on the theory of zero product determined algebras which is surveyed in the recent book \cite{zpdbook}.

In Section \ref{s2}, we show that 
a surjective additive map $T\colon A\to B$  satisfying \eqref{condf} 
is a  weighted Jordan homomorphism provided that 
 the ring $A$  is additively spanned by Jordan products of its idempotents and $B$ is any ring with $\frac{1}{2}$ (Theorem \ref{tidem}). The condition on idempotents is fulfilled in any matrix ring $M_n(R)$ with $n\ge 2$,  so this theorem  yields a  generalization  and completion of  the aforementioned  result of \cite{C} (Corollary \ref{c1}). In the first step of the proof of Theorem \ref{tidem}, which  is similar to that in \cite{ABEV1}, we reduce the problem to the situation where there exists an additive map
 $S:A\to B$ such that
\begin{equation}\label{s} S(x\circ y)=T(x)\circ T(y)\quad(x,y\in A). \end{equation}  
  In the second step, which  is based on  elementary  but tricky calculations, we show that \eqref{s} implies that $T$ is a weighted Jordan homomorphism.

  The condition \eqref{s} is 
  a simple, natural generalization of  the condition that a map is a (weighted) Jordan homomorphism, and we find it
  interesting in its own right. Our interest  also
  stems from the recent paper \cite{B} in which this condition
   unexpectedly occurred  when studying problems that are rather unrelated to those in this paper. We therefore believe that \eqref{s} deserves a systematic treatment. In Section \ref{s3}, we show that if $A$ is a prime ring with char$(A)\ne 2,3,5$ and $T:A\to A$
   is a bijective additive map for which there exists an additive map $S:A\to A$ such that \eqref{s} holds, then $T$ is a weighted Jordan homomorphism (Theorem \ref{mt3}). The proof is more  complex than the proof in Section \ref{s2}. It combines the results from the theory of  functional identities, the theory of polynomial identities, the  classical structure theory or rings, and linear algebra.

\section{Maps satisfying $xy=yx=0 \implies T(x) \circ T(y) = 0$}\label{s2}

The proof of the main theorem of  this section depends on some ideas presented in  the book \cite{zpdbook}. However,
we cannot refer directly to the results in this book
since it is (mostly) written in the context of algebras over fields while we wish to work in the context of rings. The following is the ring version of Theorem~2.15 (in conjunction with Proposition~1.3) and  Theorem~3.23 (in conjunction 
 with Remark~3.24) from \cite{zpdbook}.
 
 \begin{proposition}\label{pz}
 Let $A$ be a unital ring, let $B$ be an additive group, and let $\varphi:A\times A\to B$ be a biadditive map. If $A$ is generated as a ring by idempotents, then:
 \begin{enumerate}
 \item[{\rm (a)}] If $\varphi(x,y)=0$ whenever $x,y\in A$ are such that $xy=0$,  then $\varphi(x,y)=\varphi(xy,1)$ for all $x,y\in A$.
 \item[{\rm (b)}] If $\varphi$ is symmetric and $\varphi(x,y)=0$ whenever  $x,y\in A$ are such that $xy=yx=0$,  then $2\varphi(x,y)=\varphi(x\circ y,1)$ for all $x,y\in A$.
 \end{enumerate}
 \end{proposition}

The proof of (a) is literally the same as the proof of Lemma 2.2 and Theorem 2.3 from  \cite{zpdbook}. Using (a), one can prove (b) by simply following the proof of Theorem 3.23 from \cite{zpdbook}.  (We will actually  need only (b), but we stated also (a) to explain the proof of (b)).

We continue with a simple lemma which will be also needed  in  the next section.

\begin{lemma}\label{l0} Let $A$ and $B$ be unital rings and let $T:A\to B$ be a surjective  additive map satisfying
	\begin{equation}\label{txy}
	2T(x)\circ T(y) = T(x \circ y) \circ c \quad (x,  y\in A),
	\end{equation}
where $c=T(1)$. Assume that  $B$ is $2$-torsion free (i.e., $2b=0$ with $b\in B$ implies $b=0$) and denote the center of $B$ by $Z$. 
The following conditions are equivalent:
\begin{enumerate}
\item[(i)] $T$ is a  weighted Jordan homomorphism.
\item[(ii)]  $c^2\in Z$.
\item[(iii)] $c\in Z$.
\end{enumerate}
\end{lemma}

\begin{proof} (i)$\implies$(ii).  This is a consequence of 
the definition of a  weighted Jordan homomorphism.

 (ii)$\implies$(iii). 
 Let $b\in A$ be such that $T(b)=1$. Using \eqref{txy} we see that 
  $c^2\in Z$ implies   $$4[T(x),c]=[T(x\circ b)\circ c,c] = [T(x\circ b),c^2]=0\quad(x\in A)$$
  (here, as usual, $[x,y]$ stands for $xy-yx$).
  Since $T$ is surjective and $B$ is $2$-torsion free, $c\in Z$ follows.
  
(iii)$\implies$(i).  Asssuming that  $c\in Z$
it follows from \eqref{txy} that $4 = 4T(b^2)c$. As $B$ is $2$-torsion free this shows that 
  $c$ is invertible with $c^{-1}=T(b^2)$.  Since 
 \eqref{txy} implies that  $T(x)\circ T(y) = cT(x \circ y) $, 
$T$ is a weighted Jordan homomorphism.
\end{proof}

We will say that a ring $A$ is {\em additively spanned by Jordan products of its idempotents} if $A$ 
is equal to its additive subgroup generated 
by elements of the form $ e\circ f$ where $e$ and $f$ are idempotents. By saying that $B$ is a ring with $\frac{1}{2}$ we mean that $1+1$ is an invertible element in $B$; such a ring is of course $2$-torsion free.
(We remark that under the assumption that $\frac{1}{2}\in B$, the  condition \eqref{txy} is equivalent to the condition \eqref{s} pointed out in Section \ref{s1}, see the beginning of Section \ref{s3}).

We are now ready to state the main result of this section.

\begin{theorem}\label{tidem}
	Let $A$ and $B$ be  unital rings. Assume that  $A$ is additively spanned by Jordan products of its idempotents and  assume that  $\frac{1}{2}\in B$.
	If  $T:A\to B$ is a  surjective additive map such that for all $x,y\in A$,
	\begin{equation}\label{zj}
	xy=yx=0 \implies T(x) \circ T(y) = 0,
\end{equation}
then $T$ is a weighted Jordan homomorphism.
\end{theorem}

\begin{proof} 
Define $\varphi:A\times A\to B$ by $$\varphi(x,y)=T(x)\circ T(y).$$
Note that $\varphi$ is symmetric and that \eqref{zj} shows that $\varphi(x,y)=0$ whenever $xy=yx=0$. Since 
our assumption on $A$ in particular implies that  $A$ is generated by idempotents, it follows from Proposition \ref{pz}\,(b) that $2\varphi(x,y) = \varphi(x\circ y,1)$ for all $x,y\in A$. That is, \eqref{txy} holds (where $c=T(1)$).
In other words, we have    \begin{equation}\label{txyz}W(T(x\circ y)) = T(x)\circ T(y)\quad(x,y\in A),\end{equation}
where  $W \colon B \to B$ is defined by $$W(x) = \frac{1}{2} x \circ c.$$
Setting $x=y$ in \eqref{txyz} we obtain
	\begin{equation}\label{eq.WT}
	W(T(x^2))  = T(x)^2\quad (x\in A).
	\end{equation}
 From $y= \frac{1}{2}\big((y+1)^2 - y^2 - 1^2\big)$ we see 
	that $B$ is additively spanned by squares of its elements.  Therefore, 
 $W$ is surjective since $T$ is surjective. Further, \eqref{txyz} shows that   
 $$W(W(T(x\circ y))) =W(T(x)\circ T(y)) = \frac{1}{2} (T(x)\circ T(y))\circ c$$
 and hence 
 \begin{equation}\label{mm}
 [W(W(T(x\circ y))),c] = \frac{1}{2}[T(x)\circ T(y),c^2]\quad (x,y\in A).
 \end{equation}

	Let $e \in A$ be an idempotent. By \eqref{eq.WT},
	\begin{equation} \label{eq.TE}
	\dfrac{1}{2} T(e) \circ c = W(T(e)) = W(T(e^2)) = T(e)^2.
	\end{equation}
	This implies that $[T(e),T(e) \circ c] =0$, i.e.,
	$
	[T(e)^2, c] = 0.
	$
Hence,  \eqref{eq.TE} shows that $[T(e) \circ c, c]=0$, i.e.,
	$$
	[T(e), c^2] = 0.
	$$
Together with \eqref{mm}, this yields
$$[W(W(T(e \circ f))), c]=0$$
for all idempotents $e$ and $f$. Since $W$ and $T$ are surjective, our assumption on $A$ implies that $c$ belongs to the center of $B$. The desired conclusion that  $T$ is a weighted Jordan homomorphism now follows from
Lemma \ref{l0}. 
\end{proof}

The following corollary generalizes   \cite[Theorem 1.1]{C}; in particular, it shows that the assumption that $n\ge 4$ in this theorem is redundant.

\begin{corollary}\label{c1}Let $R$ be a  unital ring with $\frac{1}{2}$ and let $A=M_n(R)$, $n\ge 2$. If  a surjective additive map $T:A\to A$ satisfies \eqref{zj} (in particular, if $T$ preserves zero Jordan products), then $T$ is a weighted Jordan homomorphism.
\end{corollary}
\begin{proof}By $e_{ij}$ we denote the standard matrix units and by $xe_{ij}$ the matrix whose $(i,j)$ entry is $x\in R$ all other entries are $0$. Of course, each $e_{ii}$ is an idempotent. Let $i\ne j$.  Note  that
 $$xe_{ij} + e_{ii}\quad\mbox{and}\quad x(e_{ii}+e_{ji}) + (1-x)(e_{ij}+e_{jj})$$ are idempotents and
	\[
	\begin{split}
		&xe_{ij} = (xe_{ij}+e_{ii}) - e_{ii}, \\
		&xe_{ii} = \dfrac{1}{2}\big(\big( x(e_{ii}+e_{ji}) + (1-x)(e_{ij}+e_{jj}) \big) \circ e_{ii} - xe_{ji} - (1-x)e_{ij} \big).
	\end{split}
	\]
	Since $2e= e\circ e$ for every idempotent $e$ and since $\frac{1}{2}\in R$ (and so
	$xe_{ij}  = 2(\frac{1}{2}xe_{ij})$)
	 it follows that
	 $A$ is additively spanned by Jordan products of idempotents. Thus, Theorem \ref{tidem}  applies.
	 \end{proof}
	 
\begin{remark}
The assumption  that $\frac{1}{2}\in R$ cannot be removed. Indeed, if $R$ is a ring with char$(R)=2$, then
any map $T:A\to A$ of the form $T(x)= x + \lambda(x)1$, where $\lambda:A\to Z$ is an additive map,  satisfies \eqref{zj}. It is easy to find examples where such a map is surjective but is  not a weighted Jordan homomorphism.
\end{remark}

\begin{remark} Every
Jordan homomorphism on a matrix ring $M_n(R)$ is the sum of a homomorphism and an antihomomorphism \cite{JR}. The result of Corollary \ref{c1} can thus be stated as that $T(x)=c\big(\Phi_1(x)+\Phi_2(x)\big)$ where $\Phi_1$ is a homomorphism and $\Phi_2$ is an antihomomorphism.
\end{remark}

The matrix ring $M_n(R)$ is our basic and motivating example of a ring satisfying the condition of Theorem \ref{tidem}
regarding  idempotents. However, there are other examples.

\begin{example}
	Let $A$ and $B$ be unital  rings with $\frac{1}{2}$ that are additively spanned by Jordan products of idempotents. It is an easy exercise to show that the triangular ring $\tri(A,M,B)$, where $M$ is a unital 
	 $(A,B)$-bimodule, is also  additively spanned by Jordan products of idempotents (compare \cite[Corollary 2.5]{zpdbook}).
	\end{example}

\section{Pairs of maps satisfying $S(x^2)=T(x)^2$}\label{s3}

Until further notice, we assume that $A$ is a {\em unital prime  ring}  with char$(A)\ne 2$ 
and $S,T:A\to A$ are additive maps satisfying 
\begin{equation} \label{ena}S(x^2)=T(x)^2\quad(x\in A).\end{equation}
The standard linearization trick shows that \eqref{ena} is equivalent to 
\begin{equation} \label{dva}S(x\circ y)=T(x)\circ T(y) \quad(x,y\in A).\end{equation}
We assume that  {\em $T$ is bijective}.  Our goal is to prove that, under some  additional restrictions on char$(A)$ which will be imposed later,  $T$ is a weighted Jordan homomorphism.

The center of $A$ will be denoted by $Z$.  Further, we denote $c=T(1)$ and $b=T^{-1}(1)$. 
Note that \eqref{dva} shows that
 \begin{equation*}\label{ec} 2T(x)= S(x\circ b) \quad(x\in A).\end{equation*}
and
 \begin{equation}\label{ea} 2S(x)=T(x)\circ c \quad(x\in A),\end{equation}
and that \eqref{dva} and \eqref{ea} yield
 \begin{equation}\label{eb} T(x\circ y)\circ c=2T(x)\circ T(y) \quad(x,y\in A).\end{equation}
 Thus, \eqref{ena} is just a small variation of the condition \eqref{eb} that was already studied in the preceding section. Under the presence of the element $\frac{1}{2}$,  the two conditions are equivalent.


We need some more notation. By $C$ we denote the {\em extended centroid} $C$ of $A$. Recall that $C$ is a field containing the center $Z$ (see \cite[Section 7.5]{INCA} for details). 
Let $x\in A$. We write $\deg(x)=n$ if $x$ is algebraic of degree $n$ over $C$, and $\deg(x)=\infty$ if $x$ is not algebraic over $C$.
Set $\deg(A)=\sup\{\deg(x)\,|\, x\in A\}$. It is well known that the condition that $\deg(A)  < \infty$ is equivalent to the condition that $A$ is a PI-ring.

Our first  lemma was essentially proved in \cite{C}. More precisely, noticing that
\eqref{dva} implies $$T(y)\circ T(xyx) = T(x)\circ T(yxy)$$ we see that this lemma is evident from the proof 
 of \cite[Theorem 2.4]{C} along with a basic result on functional identities which states that a prime ring  $A$ is a  $d$-free
 subset of $Q=Q_{ml}(A)$, the maximal left ring of quotients of $A$, 
   if and only if $\deg(A) \ge d$ \cite[Corollary 5.12]{FIbook}. Therefore, we state it without proof.
 
 \begin{lemma}\label{l1} If $\deg(A)\ge 4$, then $T$ is a weighted Jordan homomorphism.\end{lemma}
 
It should be emphasized that Lemma \ref{l1} covers the case where 
$\deg(A)=\infty$. 
We thus only need to consider the case where $A$ is a PI-ring with $\deg(A) < 4$.
The   $\deg(A)=1$ case is trivial.

\begin{lemma}\label{ldeg1}
If $\deg(A)=1$, then $T$ is a weighted Jordan homomorphism.
\end{lemma}

\begin{proof}
The condition that  $\deg(A)=1$ means that $A$ is commutative, so $c\in Z$ automatically holds and we may apply Lemma \ref{l0}.
\end{proof}

 Hence, there are only two cases left: $\deg(A)=2$ and $\deg(A)=3$. 
The rings that remain to be considered are thus very specific.
   However,  for problems that can be solved by means of  functional identities, the low  degree situations are usually the more difficult ones. 
 
 In our next lemma  we will not yet  need the degree restriction. Its proof is also  based on functional identities. The reader is referred to \cite{FIbook} for the explanation of some notions that will be used.

 \begin{lemma}\label{l2} There exists a  ring monomorphism $\mu:Z\to Z$ such that
 $$T(zx)=\mu(z) T(x)\quad (z\in Z, x\in A),$$
  $$S(zx)=\mu(z) S(x)\quad (z\in Z, x\in A).$$
\end{lemma}

\begin{proof} Fix $z\in Z$.
Since $zx\circ y = x \circ zy$ for all $x,y\in A$ it follows from \eqref{dva} that
$T(zx)\circ T(y) = T(x)\circ T(zy)$, that is,
\begin{equation}\label{tt} T(zx)T(y) - T(zy)T(x)+ T(y)T(zx) - T(x)T(zy) =0 \quad(x,y\in A).\end{equation}  
In view of Lemma \ref{ldeg1}, the lemma is trivial if $A$ is commutative. We may thus assume that $A$ is not commutative, which implies that it is a $2$-free subset of  $Q$  \cite[Corollary 5.12]{FIbook}. Hence, applying \cite[Theorem 4.3]{FIbook} to 
\eqref{tt} we see that there exist uniquely determined $p_1,p_2\in Q$ and maps $\lambda_1,\lambda_2:A\to C$ such that
\begin{align}
T(zx) &= T(x)p_1 + \lambda_1(x)\quad(x\in A),\label{a1}\\
-T(zy) &= T(y)p_2 + \lambda_2(y)\quad(y\in A),\label{a2}\\
T(zx) &= -p_2T(x) - \lambda_1(x)\quad(x\in A),\label{a3}\\
-T(zy) &= - p_1T(y) - \lambda_2(y)\quad(y\in A)\label{a4}.
\end{align}
Comparing \eqref{a1} and \eqref{a4} we obtain
 $$T(x)p_1 - p_1T(x) =  \lambda_2(x) - \lambda_1(x) \in C\quad(x\in A).$$
Again using the  $2$-freeness of $A$ along with \cite[Theorem 4.3]{FIbook} it follows that $\lambda_1 = \lambda_2$ and 
$p_1\in C$ (similarly \eqref{a2} and \eqref{a3} show that $p_2\in C$). Next, comparing \eqref{a1} and \eqref{a2} we obtain
 $$T(x)(p_1 + p_2) = -\big(\lambda_1(x) + \lambda_2(x)\big)\in C\quad(x\in A),$$
 and so, again by the  $2$-freeness and \cite[Theorem 4.3]{FIbook}, $\lambda_1 + \lambda_2 =0$ (and $p_1 = -p_2$). 
 Since we have shown above that  $\lambda_1 = \lambda_2$ and since char$(A)\ne 2$ by assumption, it follows that $\lambda_1=0$. 
 Thus, $T(zx)= T(x)p_1$ for all $x\in A$.  
As above, let $b\in A$  be such that  $T(b)=1$. From $p_1=  T(b)p_1 = T(zb)$ we see that $p_1\in C\cap A=Z$.
 Defining
 $\mu(z)=p_1\in Z$ we thus have  $T(zx)=\mu(z) T(x)$ for all $x\in A$. From \eqref{ea} we see that this immediately implies that $S(zx)=\mu(z) S(x)$ holds too.
 
For any $z_1,z_2\in Z$, we have
$$\mu(z_1+z_2)= T\big((z_1+z_2)b\big) = T(z_1b)+ T(z_2b) = \mu(z_1) + \mu(z_2)$$
and
 $$\mu(z_1z_2)= T(z_1z_2b) = \mu(z_1)T(z_2b)= \mu(z_1)\mu(z_2),$$
 so $\mu$ is a ring endomorphism. Since $T$ is injective we see from 
 $\mu(z)=T(zb)$ that $\mu$ is injective too.
\end{proof}

From now on we assume that $\deg(A)$ equals $2$ or $3$. In particular, 
 $\deg(A) 
<\infty$, which implies that $Q_Z(A)$, the ring of central quotients of $A$, is a finite-dimensional central simple algebra over the field of quotients of $Z$. This is the content of Posner's Theorem, see \cite[Theorem 7.58]{INCA}. The elements of the ring $Q_Z(A)$ can be written as $z^{-1}x$ where 
$z\in Z\setminus{\{0\}}$ and $x\in A$.

\begin{lemma}\label{lred}There exist  additive maps $\mathcal{S},\mathcal{T}\colon Q_Z(A)\to Q_Z(A)$ such that   $\left.\mathcal{S}\right|_{A} = S$, $\left.\mathcal{T}\right|_{A} = T$, and
 $\mathcal{S}(q^2)  = \mathcal{T}(q)^2$ for every $q\in Q_Z(A)$. 
\end{lemma}

\begin{proof}For any $z\in Z\setminus{\{0\}}$ and $x\in A$, define $$\mathcal{S}(z^{-1}x) = \mu(z)^{-1}S(x).$$ Assume that $z,z'\in Z\setminus{\{0\}}$ and $x,x'\in A$ are such that  $z^{-1}x = z'^{-1}x'$. Then $z'x = zx'$ and hence
$\mu(z')S(x)=\mu(z)S(x')$, that is, $ \mu(z)^{-1}S(x) =  \mu(z')^{-1}S(x')$. This shows that $\mathcal S$ is well-defined. It is clear that  $\left.\mathcal{S}\right|_{A} = S$. Let  $z,w\in Z\setminus{\{0\}}$ and $x,y\in A$. Then 
\begin{align*}&\mathcal{S}(z^{-1}x + w^{-1}y)=\mathcal{S}\big((zw)^{-1}(wx+ zy)\big)=\mu(zw)^{-1} S(wx + zy)\\
=& \mu(z)^{-1} \mu(w)^{-1} \big(\mu(w)S(x) + \mu(z)S(y)\big) = \mathcal{S}(z^{-1}x) + \mathcal{S}(w^{-1}y), \end{align*}
so $\mathcal S$ is additive.

Similarly we see that 
$$\mathcal{T}(z^{-1}x) = \mu(z)^{-1}T(x)$$
is a well-defined additive map which extends $T$. Finally, 
$$\mathcal{S}\big((z^{-1}x)^2\big)=\mathcal{S}(z^{-2}x^2)=\mu(z^2)^{-1}S(x^2)= \mu(z)^{-2}T(x)^2= \mathcal{T}(z^{-1}x)^2,$$ 
which proves that  $\mathcal{S}(q^2)  = \mathcal{T}(q)^2$ for every $q\in Q_Z(A)$. 
\end{proof}

Lemma \ref{lred} shows that there is no loss of generality in assuming that  $A=Q_Z(A)$ is a central simple algebra 
such that $\deg(A)=2$ or $\deg(A) =3$, or equivalently, $\dim_Z(A)=4$ or $\dim_Z(A)=9$ (see \cite[Theorem C.2]{FIbook}). Furthermore,
 in light of Corollary  \ref{c1} we may assume that $A$ is not a ring of $n\times n$ matrices, $n\ge 2$, over some ring, and hence, by the classical Wedderburn's structure theorem, we may assume that {\em  $A$ is a division ring}.

\begin{lemma}\label{ldeg2}
If $\deg(A)=2$, then $T$ is a  weighted Jordan homomorphism.
\end{lemma}

\begin{proof}
Since $\deg( A) = 2$, there exist an additive map $\tau:A\to Z$ and a biadditive map $\delta:A^2\to Z$ such that
\begin{equation}\label{t2}
	x^2 = \tau(x) x+ \delta(x,x) 
\end{equation}
for every $x\in A$ 
(see \cite[Corollary C.3]{FIbook}). We may assume that $\delta$ is symmetric, since otherwise we replace it by $(x,y)\mapsto \frac{1}{2}\big(\delta(x,y) + \delta(y,x)\big)$. Linearizing \eqref{t2} we obtain that for all $x,y\in A$,
\begin{equation*}\label{tt2}
	x\circ y = \tau(x) y+\tau(y)x+ 2\delta(x,y).
\end{equation*}
From \eqref{ea} we thus see that
\begin{equation}\label{st}2S(x)= \tau(T(x))c + \tau(c)T(x) + 2\delta(T(x),c)\end{equation}
for all $x\in A$.
Next, applying $S$ to \eqref{t2} and using \eqref{ena}  we obtain
$$T(x)^2 = \mu(\tau(x))S(x)  + \mu\big(\delta(x,x)\big)c^2.$$
Applying \eqref{t2} and \eqref{st} we see that this can be rewritten as
\begin{align}&\tau(T(x))T(x) + \delta(T(x),T(x))\nonumber\\ \label{emil}=&\frac{1}{2}  \mu(\tau(x))\tau(T(x))c + \frac{1}{2}  \mu(\tau(x)) \tau(c)T(x) \\
&+  \mu(\tau(x))\delta(T(x),c) + \mu\big(\delta(x,x)\big)c^2.\nonumber\end{align}
Commuting this identity with $c$ we obtain
$$\Big(\tau(T(x))-\frac{1}{2}  \mu(\tau(x))\tau(c) \Big)[T(x),c]=0$$
for all $x\in A$. Accordingly, for each $x\in A$ we have either
\begin{equation}\label{eit} \tau(T(x))=\frac{1}{2}  \mu(\tau(x))\tau(c)\end{equation}
or $[T(x),c]=0$.  The set of all $x\in A$ that satisfy one of these two conditions is an additive subgroup of $A$. 
Since a group cannot be the union of two proper subgroups, one of the two conditions must hold for every $x\in A$. If $[T(x),c]=0$ for every $x\in A$, then $c\in Z$ and so $T$ is a weighted Jordan homomorphism by Lemma \ref{l0}. We may thus assume that \eqref{eit} holds for every $x\in A$.

Note that \eqref{emil} along with $c^2 =\tau(c)c+\delta(c,c)$ now shows that 
$$\Big(\frac{1}{2}  \mu(\tau(x))\tau(T(x)) +  \mu\big(\delta(x,x)\big)\tau(c)\Big)c \in Z.$$
Therefore, either $c\in Z$ or
$$\frac{1}{2}  \mu(\tau(x))\tau(T(x)) +  \mu\big(\delta(x,x)\big)\tau(c) =0$$
for every $x\in A$. Assume that the latter holds. By \eqref{eit}, we can rewrite this identity as
\begin{equation}\label{bih}\Big(\frac{1}{4}  \mu(\tau(x)^2)  +  \mu\big(\delta(x,x)\big)\Big)\tau(c)=0.\end{equation}
If $\tau(c)=0$ then it follows from \eqref{eit} that $\tau(T(x))=0$   and hence $T(x)^2\in Z$ for every $x\in A$. Since $T$ is surjective, this means that $y^2\in Z$ for every $y\in A$, which leads to a contradiction that
$y= \frac{1}{2}\big((y+1)^2 - y^2 - 1\big)\in Z$ for every $y\in A$. Thus, $\tau(c)\ne 0$ and so \eqref{bih} implies that
$$\frac{1}{4}  \mu(\tau(x)^2)  +  \mu\big(\delta(x,x)\big)=0$$
for every $x\in A$. Since $\mu$ is injective it follows that $$\frac{1}{4} \tau(x)^2  =-\delta(x,x).$$ Together with \eqref{t2} this yields
$$\big(x-\frac{1}{2}\tau(x)\big)^2 = x^2 - \tau(x)x + \frac{1}{4} \tau(x)^2=x^2 - \tau(x)x - \delta(x,x)=0.$$
Since, as mentioned before the statement of the lemma, we may assume that $A$ is a division ring, this implies that $x=\frac{1}{2}\tau(x)\in Z$ for every $x\in A$, a contradiction. Therefore, $c\in Z$ and so $T$ is a weighted Jordan homomorphism
by Lemma \ref{l0}.
\end{proof}

The case where $\deg(A)=3$ is more involved. To handle it, we need
the following 
 linear algebra lemma.
 
\begin{lemma}\label{m1}
	Let $K$ be an algebraically closed field with {\rm char}$(K)\ne 2,3$.  Let $t \in M_3(K)$. Then the set  $$S_y=\{t, t \circ y, (t \circ y) \circ y, ((t \circ y) \circ y) \circ y \}$$ is linearly dependent for every $y \in M_3(K)$ if and only if  $t$ is a scalar matrix.
\end{lemma}

\begin{proof} The ``if'' part  follows from the Cayley-Hamilton Theorem. To prove the ``only if'' part, assume that $t$ is not a scalar matrix. 
	Our goal is to find a $y \in A$ such that the set $S_y$ is linearly independent. Since $K$ is algebraically closed, we may assume that $t$ is in the Jordan normal form. 
	
	We divide the proof into four cases.
	
	\begin{itemize}
		\item[{\bf 1.}] 
		 Assume that
		\[
		t = 
		\begin{pmatrix}
			\lambda & 1 & 0 \\
			0 & \lambda & 1 \\
			0 & 0 & \lambda
		\end{pmatrix}
		\]
		where $\lambda \in K$. If
		\[
		y = 
		\begin{pmatrix}
			0 & 0 & 0 \\
			1 & 0 & 0 \\
			0 & 1 & 0
		\end{pmatrix},
		\]
		then $t \circ y$, $(t \circ y) \circ y $, $(t \circ y) \circ y) \circ y$ are
		\begin{align*}
		\begin{pmatrix}
			1 & 0 & 0 \\
			2\lambda & 2 & 0 \\
			0 & 2\lambda & 1
		\end{pmatrix}, 
		\begin{pmatrix}
			0 & 0 & 0 \\
			3 & 0 & 0 \\
			4\lambda & 3 & 0
		\end{pmatrix}, 
		\begin{pmatrix}
			0 & 0 & 0 \\
			0 & 0 & 0 \\
			6 & 0 & 0
		\end{pmatrix},
		\end{align*} respectively.
		It is easy to check that $S_y$ is linearly independent.
		
		\item[\bf  2.] Assume that
		\[
		t = 
		\begin{pmatrix}
			\lambda_1 & 1 & 0 \\
			0 & \lambda_1 & 0 \\
			0 & 0 & \lambda_2
		\end{pmatrix}
		\]
		where $\lambda_1, \lambda_2 \in K$. If 
		\[
		y = 
		\begin{pmatrix}
			0 & 0 & 0 \\
			1 & 0 & 0 \\
			0 & 1 & 0
		\end{pmatrix},
		\]
			then $t \circ y$, $(t \circ y) \circ y $, $(t \circ y) \circ y) \circ y$ are
		\begin{align*}
		\begin{pmatrix}
			1 & 0 & 0 \\
			2\lambda_1 & 1 & 0 \\
			0 & \lambda_1+\lambda_2 & 0
		\end{pmatrix},
		\begin{pmatrix}
			0 & 0 & 0 \\
			2 & 0 & 0 \\
			3\lambda_1 + \lambda_2 & 1 & 0
		\end{pmatrix}, 
		\begin{pmatrix}
			0 & 0 & 0 \\
			0 & 0 & 0 \\
			3 & 0 & 0
		\end{pmatrix},
		\end{align*}respectively.
		Again, it is easy to see that $S_y$ is linearly independent.
		
			\item[\bf  3.] Assume that
		\[
		t = 
		\begin{pmatrix}
			\lambda_1 & 0 & 0 \\
			0 & \lambda_2 & 0 \\
			0 & 0 & \lambda_3
		\end{pmatrix}
		\]
		where $\lambda_1,\lambda_2, \lambda_3 \in K$ are not all equal and $\lambda_1 + \lambda_2 + \lambda_3 \neq 0$. If
		\[
		y = 
		\begin{pmatrix}
			0 & 0 & 1 \\
			1 & 0 & 0 \\
			0 & 1 & 0
		\end{pmatrix},
		\]
			then $t \circ y$, $(t \circ y) \circ y $, $(t \circ y) \circ y) \circ y$ are
		\begin{align*}
			&\begin{pmatrix}
				0 & 0 & \lambda_1 + \lambda_3 \\
				\lambda_1 + \lambda_2 & 0 & 0 \\
				0 & \lambda_2 + \lambda_3 & 0
			\end{pmatrix}, \\
			&\begin{pmatrix}
				0 & \lambda_1 + \lambda_2 + 2\lambda_3 & 0 \\
				0 & 0 & 2\lambda_1 + \lambda_2 + \lambda_3 \\
				\lambda_1 + 2\lambda_2 + \lambda_3 & 0 & 0
			\end{pmatrix}, \\
			&\begin{pmatrix}
				2\lambda_1 + 3\lambda_2 + 3\lambda_3 & 0 & 0 \\
				0 & 3\lambda_1 + 2\lambda_2 + 3\lambda_3 & 0 \\
				0 & 0 & 3\lambda_1 + 3\lambda_2 + 2\lambda_3
			\end{pmatrix},
		\end{align*} respectively.
		A slightly more tedious but still elementary argument shows that $S_y$ is linearly independent in this case too.

		
			\item[\bf 4.] We now consider the last remaining case where
		\[
		t = 
		\begin{pmatrix}
			\lambda_1 & 0 & 0 \\
			0 & \lambda_2 & 0 \\
			0 & 0 & -\lambda_1 - \lambda_2
		\end{pmatrix}
		\]
		with $\lambda_1, \lambda_2 \in K$ and $\lambda_1 \neq 0$. If
		\[
		y = 
		\begin{pmatrix}
			1 & 0 & 0 \\
			1 & 0 & 0 \\
			0 & 1 & 0
		\end{pmatrix},
		\]
			then $t \circ y$, $(t \circ y) \circ y $, $(t \circ y) \circ y) \circ y$ are
		\begin{align*}
			& 
			\begin{pmatrix}
				2\lambda_1 & 0 & 0 \\
				\lambda_1 + \lambda_2 & 0 &0 \\
				0 & -\lambda_1 & 0
			\end{pmatrix}, \\
			& 
			\begin{pmatrix}
				4\lambda_1 & 0 & 0 \\
				3\lambda_1 + \lambda_2 & 0 & 0 \\
				\lambda_2 & 0 & 0
			\end{pmatrix}, \\
			& 
			\begin{pmatrix}
				8\lambda_1 & 0 & 0 \\
				7\lambda_1 + \lambda_2 & 0 & 0 \\
				3\lambda_1 + 2\lambda_2 & 0 & 0
			\end{pmatrix},
		\end{align*}
		respectively. One easily checks that $S_y$ is linearly independent.
			\end{itemize}	We have thus proved that for each $t$ that is not a scalar matrix  there is a matrix $y$ such that  $S_y$ is linearly independent.\end{proof}

 We are now ready to tackle the $\deg(A)=3$ case.
 
 \begin{lemma}\label{l9} If $\deg(A)=3$ and {\rm char}$(A)$ is also different from $3$ and $5$, then $T$ is a weighted Jordan homomorphism.
 \end{lemma}

\begin{proof} We start similarly as in the proof of Lemma \ref{ldeg2}.
As  $\deg( A) = 3$, there exist  symmetric multiadditive maps $\tau \colon A \to Z$,  $\alpha \colon A^2 \to Z$ and $\delta \colon A^3 \to Z$ such that
\[
	x^3 = \tau(x) x^2 + \alpha(x,x)x + \delta(x,x,x) 
\]
for every $x\in A$,
with $\delta(x,x,x)$ being the reduced norm of $x$. Since, as pointed out above, Corollary \ref{c1} enables us to assume that $A$ is a division ring, $\delta(x,x,x)\ne 0$ for every nonzero $x\in A$. In particular,
$$\gamma = \mu\big(\delta(b,b,b)\big) \ne 0$$
(here, as always, $b=T^{-1}(1)$).

Let $x \in A$. From \eqref{eb} we see that
$T(x^2) \circ T(x) =  T(x^3) \circ c.$
Hence,
\[
\begin{split}
	T(x^2) \circ T(x) 
	 = &T\big(\tau(x)x^2 + \alpha(x,x)x + \delta(x,x,x)\big) \circ c \\
	=& \mu(\tau(x)) T(x^2) \circ c + \mu(\alpha(x,x)) T(x) \circ c \\&+ \mu(\delta(x,x,x)) c \circ c.
\end{split}
\]
Since, again by \eqref{eb}, $T(x^2) \circ c = T(x)\circ T(x)$, it follows that
\begin{equation}
\big( T(x^2) - \mu(\tau(x)) T(x) - \mu(\alpha(x,x)) c \big) \circ T(x) = 2\mu(\delta(x,x,x)) c^2.
\label{dd}
\end{equation}

Define $f \colon A^2 \to A$ by
\[
f(x,y) = \dfrac{1}{2} T(x \circ y) - \dfrac{1}{2} \mu(\tau(x)) T(y) - \dfrac{1}{2} \mu(\tau(y)) T(x) - \mu(\alpha(x,y))c.
\]
Observe that $f$ is a symmetric biadditive map  which, by \eqref{dd}, satisfies
\begin{equation}\label{ss}
f(x,x) \circ T(x) = 2\mu(\delta(x,x,x)) c^2.
\end{equation}
Linearizing this identity we obtain
\begin{equation}\label{eq.fxyTz}
f(x,y) \circ T(z) + f(z,x) \circ T(y) + f(y,z) \circ T(x) = 6\mu(\delta(x,y,z)) c^2.
\end{equation}
By \eqref{ss}, 
\[
f(b,b) = \dfrac{1}{2} f(b,b) \circ T(b) = \mu(\delta(b,b,b)) c^2 = \gamma c^2.
\]
Putting  $y=z=b$ in  \eqref{eq.fxyTz}   we obtain
\[
4f(x,b) +  \gamma c^2 \circ T(x) = 6 \mu(\delta(x,b,b)) c^2.
\]
Next, applying \eqref{eq.fxyTz} with $y=x$ and $z=b$ we arrive at
\[
f(x,x) + f(x,b) \circ T(x) = 3\mu(\delta(x,x,b)) c^2.
\]
The last two identities yield
$$
f(x,x) =  \Big( \frac{\gamma}{4}  c^2 \circ T(x) - \frac{3}{2} \mu(\delta(x,b,b)) c^2\Big)\circ T(x) + 3\mu(\delta(x,x,b)) c^2.
$$
Returning to \eqref{ss}, we now have
\begin{align*} \frac{\gamma}{4}  \big((c^2 \circ T(x))\circ  T(x)\big)\circ T(x)  - &\frac{3}{2} \mu(\delta(x,b,b)) (c^2 \circ T(x))\circ T(x)\\
+&  3\mu(\delta(x,x,b)) c^2\circ T(x)= 2\mu(\delta(x,x,x)) c^2.
\end{align*}
Since $\gamma\ne 0$ and $T$ is surjective, this shows that for each $y\in A$, the set $$\{c^2, c^2 \circ y, (c^2 \circ y) \circ y, ((c^2 \circ y) \circ y) \circ y \}$$ is linearly dependent. We will now use the  fact known from the theory of polynomial identities that
the linear dependence can be characterized through a special identity, see \cite[Theorem 7.45]{INCA}. Denoting by  $c_4$  the 4th {\em Capelli polynomial}, this theorem implies that
$$c_4\big( c^2, c^2 \circ y, (c^2 \circ y) \circ y, ((c^2 \circ y) \circ y) \circ y , x_1,x_2,x_3\big)=0$$
for all $y,x_1,x_2,x_3\in A$. Since $c_4$ is multilinear, the linearization of this identity gives
$$\sum_{\sigma\in S_{6}}c_4\big( c^2, c^2 \circ y_{\sigma(1)}, (c^2 \circ y_{\sigma(2)}) \circ y_{\sigma(3)}, ((c^2 \circ y_{\sigma(4)}) \circ y_{\sigma(5)}) \circ y_{\sigma(6)} , x_1,x_2,x_3\big)=0$$
for all $x_i,y_j\in A$, $i=1,2,3$, $j=1,\dots,6$. Let $K$ be the algebraic closure of $Z$ and let $\overline{A}= K\otimes_Z A$. 
Since each $x_i$ and each $y_j$ occurs linearly in the last identity, it follows that $t=1\otimes c^2\in \overline{A}$ satisfies
$$\sum_{\sigma\in S_{6}}c_4\big( t, t \circ y_{\sigma(1)}, (t \circ y_{\sigma(2)}) \circ y_{\sigma
(3)}, ((t \circ y_{\sigma(4)}) \circ y_{\sigma(5)}) \circ y_{\sigma(6)} , x_1,x_2,x_3\big)=0$$
for all $x_i,y_j\in \overline{A}$, $i=1,2,3$, $j=1,\dots,6$. Take each $y_i$ to be equal to $y$. Our characteristic assumption implies that $6!u=0$ with $u\in \overline{A}$ implies $u=0$, so we have
$$c_4\big( t, t \circ y, (t \circ y) \circ y, ((t \circ y) \circ y) \circ y , x_1,x_2,x_3\big)=0$$
for all $x_i,y\in A$, $i=1,2,3$.  We may now again use  \cite[Theorem 7.45]{INCA}, this time in the opposite direction, to conclude that the set
$$\{t, t \circ y, (t \circ y) \circ y, ((t \circ y) \circ y) \circ y \}$$ is linearly dependent for every $y \in \overline{A}$.
Since $\overline{A}\cong M_3(K)$ \cite[Theorem 4.39]{INCA},
 Lemma \ref{m1} shows that $t$ lies in the center of $\overline{A}$. Consequently, $c^2\in Z$ and so Lemma \ref{l0} tells us that $T$ is a weighted Jordan homomorphism.\end{proof}

We can now state  the main result of this section.

\begin{theorem}\label{mt3}
Let $A$ be a  prime ring with {\rm char}$(A)\ne 2,3,5$ and let $S,T:A\to A$ be additive maps such that $S(x^2)=T(x)^2$ for every $x\in A$. If $T$ is bijective, then it is a weighted Jordan homomorphism. 
\end{theorem}

\begin{proof} Apply Lemmas \ref{l1}, \ref{ldeg1}, \ref{ldeg2}, and \ref{l9}.
\end{proof}

\begin{remark}
The conclusion of Theorem \ref{mt3} is that $T(x)=c\Phi(x)$ where $\Phi$ is a Jordan automorphism of $A$. It is well known that, since $A$ is prime, $\Phi$ is either an automorphism or an antiautomorphism \cite{Her}.
\end{remark}
 
\begin{remark}\label{rtr}
The injectivity of $T$ was used only once in the proof of Theorem \ref{mt3}, that is, when showing that $\mu$ is injective. If $Z$ is a field, in particular if $A$ is simple, then $\mu$ is automatically injective and so we may weaken the assumption that  $T$ is bijective to $T$ being only surjective.
\end{remark}

In our final result we return to the condition studied in Section \ref{s2}.

\begin{corollary}\label{c12}Let $A$ be a  unital  simple ring with {\rm char}$(A)\ne 2,3,5$. If $A$ contains a nontrivial idempotent, then 
every  surjective additive map $T:A\to A$ such that for all $x,y\in A$,
	$$xy=yx=0 \implies T(x) \circ T(y) = 0,$$
 is a weighted Jordan homomorphism.
\end{corollary}

\begin{proof}
It is well known that the  existence of one nontrivial idempotent in a simple ring $A$ implies that $A$ is generated by idempotents \cite[Corollaries on p.\ 9 and p.\ 18]{Her1}. We can therefore repeat the argument from the beginning of the proof of Theorem \ref{tidem} and conclude that $T$ satisfies condition \eqref{eb}, which is of course an equivalent  version of the condition $S(x^2)=T(x)^2$ studied in Theorem \ref{mt3}. As pointed out in Remark \ref{rtr}, in this setting the injectivity of $T$ is not needed for reaching the conclusion that $T$ is a weighted Jordan homomorphism.
\end{proof}

Theorems \ref{tidem} and \ref{mt3} show that in quite general rings,
 weighted Jordan homomorphisms
 are the only bijective additive  maps $T$ with the property that  $S(x^2)=T(x)^2$ for some additive map $S$.
 We conclude the paper with an example showing that there exist rings in which this does not hold.

\begin{example}
Let $A$ be the Grassmann algebra in two generators over a field $F$ with char$(F)\ne 2$. That is, $A$ is the $4$-dimensional algebra with basis $1,u,v,uv$ where $u^2=v^2=uv+vu=0$. For each $x\in A$, let $\lambda(x)$ be the element in $F$ satisfying
$x-\lambda(x)1\in {\rm span}\{u,v,uv\}$. Note that $x\mapsto \lambda(x) $ is an algebra homomorphism from $A$ to $F$ and that $x\circ u=2\lambda(x)u$ for every $x\in A$.
Define $S,T:A\to A$ by
 $$T(x)= x  + \lambda(x)u,\quad S(x)= x  + 2\lambda(x)u.$$ 
Then $S$ and $T$ are linear maps, $T$ is bijective, and 
\begin{align*}
S(x^2)=& x^2  + 2\lambda(x^2)u = x^2 + 2\lambda(x)^2u\\ = &x^2 + \lambda(x) (x\circ u)= (x+ \lambda(x)u)^2 = T(x)^2
\end{align*}
for every $x\in A$.
However, $T(1)= 1+u$ is not a central element, so $T$ is not a weighted Jordan homomorphism.
\end{example}

\end{document}